\documentclass[12pt]{amsart}

\usepackage{amssymb, latexsym}
\usepackage{float}

\def\mf{\mathsf}

\def\mr{\mathrm}
\def\p{\varphi}

\def\mbR{\mathbb{R}}
\def\mbZ{\mathbb{Z}}

\def\mbD{\mathbb{D}}

\newcommand{\mc}{\mathcal}

\newcommand{\uIo}{\underline{I}^{o}}
\newcommand{\ovT}{\overline{T}}

\theoremstyle{plain}
\newtheorem{theorem}{Theorem}[section]
\newtheorem{corollary}[theorem]{Corollary}
\newtheorem{proposition}[theorem]{Proposition}
\newtheorem{lemma}[theorem]{Lemma}

\theoremstyle{definition}

\newtheorem{remark}[theorem]{Remark}
\newtheorem{example}[theorem]{Example}

\newtheorem{question}[theorem]{Question}

\numberwithin{equation}{section}

\usepackage
{hyperref}

\usepackage[dvipsnames]{color}

\begin{document}

\textcolor[rgb]{0,0,1}{}

\title[Geometry of dyadic polygons II] {Geometry of dyadic polygons II:\\ Isomorphisms of dyadic triangles}

\author[Mu\'{c}ka]{A. Mu\'{c}ka$^1$}
\address{$^1$ Faculty of Mathematics and Information Sciences\\
Warsaw University of Technology\\
00-661 Warsaw, Poland}

\author[Romanowska]{A.B. Romanowska$^2$}
\address{$^2$ Faculty of Mathematics and Information Sciences\\
Warsaw University of Technology\\
00-661 Warsaw, Poland}

\email{$^1$Anna.Mucka@pw.edu.pl\phantom{,}}
\email{$^2$Anna.Romanowska@pw.edu.pl\phantom{,}}



\date{August, 2025}

\begin{abstract}
This paper is the second part of a two-part paper investigating the structure and properties of dyadic polygons. A dyadic polygon is the intersection of the dyadic subplane $\mbD^2$ of the real plane $\mbR^2$ and a real convex polygon with vertices in the dyadic plane. Such polygons are described as subreducts (subalgebras of reducts) of the affine dyadic plane $\mbD^2$, or equivalently as commutative, entropic and idempotent groupoids under the binary operation of arithmetic mean.

The first part of the paper contained a new classification of dyadic triangles, considered as
such groupoids, and a characterization of dyadic triangles with a pointed vertex.
This second part investigates isomorphisms of dyadic triangles, and provides a full
classification of their isomorphism types.
\end{abstract}

\maketitle


\section{Introduction}

This paper is a continuation of ``Geometry of dyadic polygons I: The structure of dyadic triangles'' \cite{MR25}, which provided a new classification of dyadic triangles, based on the fact that each is isomorphic to a certain special dyadic triangle called a \emph{representative hat}. That paper also gave a characterization of representative hats with a pointed vertex by means of certain triples of integers. The present paper investigates automorphisms and isomorphisms of dyadic triangles, and provides a full
classification of their isomorphism types.

All the notation and results of \cite{MR25} are carried over here. The reader should consult \cite{MR25} for further background concerning dyadic polytopes, and for notation not explicitly defined here.

We now just recall some basic notation and facts.
In what follows, $\mbR$ denotes the ring of real numbers, and $\mbD  = \mbZ[1/2]$ is its subring consisting of dyadic rationals (i.e., rational numbers whose denominator is a power of $2$). The affine $\mbD$-spaces of interest here are idempotent reducts of faithful $\mbD$-modules, and subreducts of affine $\mbR$-spaces. Dyadic polygons are the intersections of the dyadic
plane and real polygons with vertices in the dyadic plane. Their one-dimensional analogs are
dyadic intervals (considered as intervals with their ends).

Real convex sets are described algebraically as certain barycentric algebras $(B,\uIo)$, subsets of a space $\mbR^n$ closed under operations of weighted means with weights from the open unit interval $I^{\circ} = ]0,1[$. Similarly, dyadic convex sets may be described as subsets of $\mbD^n$ closed under weighted means with weights from the open dyadic unit interval $I^{\circ} \cap \mbD$. Equivalently, dyadic convex sets are described as algebras with the single basic binary operation 
\[
x \circ y := xy \underline{1/2} = \frac{1}{2}(x + y)
\]
of the arithmetic mean.
The operation $\circ$ is commutative, idempotent and entropic. This allows one to define dyadic convex sets as algebras with one basic binary operation (a groupoid or magma) with the algebraic structure of a so-called \emph{commutative binary mode} ($\mc{CB}$-mode) \cite{RS02}. In particular, all dyadic intervals and dyadic triangles belong to the class of $\mc{CB}$-modes.

The paper is organized as follows. We first recall basic definitions and results from the first part concerning the structure of dyadic triangles, their representation as representative hats, and the classification and characterization of representative hats (Sections \ref{S:Dyadconvsets} and \ref{S:Triangles}). Section~\ref{S:autoreprhats} contains an analysis of automorphisms of representative hats. Finally, Section~\ref{S:isomorphisms} investigates their isomorphisms, and provides a full
classification of isomorphism types of representative hats.

As in the first part \cite{MR25}, the notation, terminology and conventions are similar to those of \cite{RS02}. The reader may wish to consult the papers referenced in \cite{MR25}, in particular, \cite{CR13}, \cite{MMR19}, \cite{MMR19a}, \cite{MMR23} and \cite{MRS11}. For more details and information on affine spaces, convex sets and barycentric algebras, we also refer the reader to the monographs \cite{RS85, RS02} and the new survey \cite{R18}. For convex polytopes, see \cite{AB83, BG03, Z95}.

\section{Dyadic convex sets}\label{S:Dyadconvsets}

First recall that affine spaces over a commutative ring $R$ (\emph{affine $R$-spaces}), where $R$ is $\mbR$ or $\mbD$, can be considered as the reducts $(A,\underline{R})$ of $R$-modules $(A,+,R)$, where $\underline{R}$ is the set of binary affine combinations
\begin{equation}\label{E:afoper}
ab\,\underline{r} = a(1-r)+br \,
\end{equation}
for all $r \in R$ and $a,b \in A$. (See~\cite[\S~5.3, \S~6.3]{RS02}.) In particular,
affine spaces over the ring $\mbD$ (affine $\mbD$-spaces) are considered here as algebras
$(A,\underline{\mbD})$ with the set $\underline{\mbD} = \{\underline{d} \mid d \in \mbD\}$ of
basic operations. The class of all affine $R$-spaces forms a variety~\cite{C75}.
In this paper we are interested in the one-dimensional affine $\mbD$-space $\mbD$, and the two-dimensional affine $\mbD$-plane $\mbD^2$.

Recall that automorphisms of the affine $\mbD$-space $\mbD^n$ form the $n$-dimensional
affine group $\mathrm{GA}(n,\mbD)$ over the ring $\mbD$, the group generated by the linear group $\mathrm{GL}(n,\mbD)$ and the group of translations of the space $\mbD^n$.

Dyadic convex sets considered here are the intersections of convex subsets $C$ of $\mbR^n$  with the subspace~$\mbD^n$, and are considered as subgroupoids $(B, \circ)$ of the reduct $(\mbD^n, \circ)$ of the affine $\mbD$-space $\mbD^n$.
In particular, a \emph{dyadic $n$-dimensional polytope} is the intersection with the dyadic space $\mbD^n$ of an $n$-dimensional real polytope whose vertices lie in the dyadic space.
For any dyadic $n$-dimensional polytope $P$,
\begin{equation}\label{E:pol}
P = \mr{conv}_{\mbR}(P) \cap \mbD^n = \mr{conv}_{\mbD}(P),
\end{equation}
where $\mr{conv}_{\mbR}(P)$ is the convex $\mbR$-hull of $P$, and $\mr{conv}_{\mbD}(P)$ is  the convex $\mbD$-hull of $P$. (Closed) intervals of $\mbD$ are one-dimensional polytopes.
Polygons of $\mbD^2$, in particular triangles, are two-dimensional polytopes.
Unlike the real case, there are infinitely many pairwise non-isomorphic dyadic intervals, and infinitely many pairwise non-isomorphic dyadic triangles.

A special role is played by dyadic simplices. Dyadic simplices may be defined in similar fashion to real simplices. An $n$-dimensional \emph{dyadic simplex} is an $n$-dimensional dyadic convex set with $n+1$ vertices, generated by these vertices. Among dyadic polygons, all dyadic intervals isomorphic to the dyadic unit interval are $1$-dimensional simplices, and all $2$-dimensional simplices are isomorphic to the dyadic triangle generated by the three elements $e_0 = (0,0)$, $e_1 = (1,0)$ and $e_2 = (0,1)$ of $\mbD^2$.

Dyadic polytopes are considered as  $\circ$-subreducts  of their affine $\mbD$-hulls. After introducing coordinate axes in the affine hull, a given polytope is located in the corresponding $\mbD$-module by providing the coordinates of its vertices. Affine $\mbD$-spaces of dimension $n$ are considered as subreducts of the $n$-dimensional real affine space $\mbR^n$, while isomorphisms of dyadic polytopes are considered as restrictions of automorphisms of their affine $\mbD$-hulls.

\section{Dyadic triangles}\label{S:Triangles}

In this section we recall basic facts from \cite{MR25} concerning dyadic intervals and dyadic triangles.

\subsection{Dyadic intervals}

Each non-trivial interval of $\mbD$, considered as a $\mc{CB}$-mode, is isomorphic
to an interval of \emph{type $k$}, a dyadic interval $\mbD_k = [0,k]$,
where $k$ is an odd positive integer. (See~\cite{MRS11}.) Two such intervals are isomorphic
precisely when their right hand ends are equal.
The interval $\mbD_1$ is generated by its ends, and the interval $\mbD_{k}$ with $k > 1$ is
generated by three (but not fewer) elements.

\subsection{Dyadic triangles and representative hats}

Each dyadic triangle contained in the plane $\mbD^2$ is isomorphic to a \emph{pointed} triangle $ABC$ in the first quadrant of the plane, located as in Figure~\ref{F:1}, and denoted  $T_{i,j,m,n}$. One ``pointed'' vertex, say $A$, is located at the origin, while the vertices $B$ and $C$ have non-negative integer coordinates.

\begin{figure}[bht]
\begin{center}
\begin{picture}(140,120)(0,0)

\put(0,20){\vector(1,0){130}}
\put(20,0){\vector(0,1){110}}
\put(20,20){\line(1,1){60}}

\put(20,80){\line(1,0){100}}
\put(10,77){$n$}
\put(80,20){\line(0,1){80}}
\put(77,8){$i$}
\put(20,100){\line(1,0){100}}
\put(10,97){$j$}
\put(120,20){\line(0,1){80}}
\put(117,8){$m$}

\put(80,80){\circle*{3}}
\put(82,69){$G$}

\put(20,20){\circle*{5}}
\put(7,7){$A$}
\put(80,100){\circle*{5}}
\put(77,107){$B$}
\put(120,80){\circle*{5}}
\put(125,76){$C$}

\thicklines

\put(20,20){\line(5,3){100}}
\put(20,20){\line(3,4){60}}
\put(120,80){\line(-2,1){40}}

\end{picture}
\end{center}
\caption{}
\label{F:1}
\end{figure}

Originally, the triangles $T_{i,j,m,n}$ were divided into three groups: \emph{right triangles} (whose shorter sides are parallel to the coordinate axes), \emph{hat triangles} (one of whose sides is parallel to a coordinate axis) and \emph{others}.

Each side of $T_{i,j,m,n}$ is a dyadic interval, so it is isomorphic
to some $\mbD_k$. If the sides of a (dyadic) triangle have respective
types $r, s, t$, then the triangle has \emph{boundary type} $(r,s,t)$, defined up to cyclic order. The triangles $T_{i,j,m,n}$ are thought of as triangles with clockwise ordered vertices $A, B, C$ and are denoted by $ABC$, determining cyclic boundary type corresponding to $(AB,BC,CA)$.  Isomorphic dyadic triangles with the same orientation of vertices have the same boundary type.

In \cite[Prop.~4.5]{MR25}, it was shown that each triangle $T_{i,j,m,n}$ is isomorphic to a triangle $T_{i',j',m',0}$ with a (positive or arbitrary) integer $i'$ and positive integers $j', m'$. Moreover, all three integers $i', j', m'$ may be choosen to be odd. (\cite[Prop.~4.11]{MR25}.

In this paper, the name \emph{hats} will be used for dyadic hat triangles $T_{i,j,m,0}$ with any integer $i$ and positive integers $j, m$. Sometimes, we will use the name \emph{right hats} for the case where $i = 0$ or $i = m$, \emph{proper hats} if $0 < i < m$, and \emph{crooked hats} for the remaining cases.
A triangle $T_{i,j,m,0}$ with odd $i, j, m$ will be called a \emph{representative hat},  and denoted by $T_{i,j,m}$.

\begin{theorem}\cite[Thm.~4.13]{MR25}\label{T:reprhats}
Each dyadic triangle in the dyadic space $\mbD^2$ is isomorphic to a representative hat.
\end{theorem}

Note that each dyadic triangle $ABC$ is isomorphic to three
representative hats with the same ordering of the vertices, and to three isomorphic hats with the
reverse ordering. These hats correspond to the six permutations of the vertices.

\subsection{Pointed representative hats}

A hat with a pointed vertex located at the
origin will be called a \emph{pointed} hat, and isomorphisms of pointed hats will be considered as isomorphisms preserving the pointed vertex
and the orientation of the vertices. This type of isomorphism will be called a \emph{pointed
oriented isomorphism}, or briefly a \emph{pointed isomorphism}.

\begin{proposition}\cite[Cor.~5.3]{MR25}\label{C:nonisomhats}
For any positive odd integers $j$ and $m$, there are $j$  (pointed) isomorphism classes of
representative hats $T_{i,j,m}$. Each class is represented by a unique $T_{i,j,m}$, where $i \in \{1,3,\dots, 2j-1\}$.
\end{proposition}

The triple $(i,j,m)$ of Proposition \ref{C:nonisomhats} is called an \emph{encoding triple}. Each encoding triple determines uniquely a class of (pointed) representative hats isomorphic to a given dyadic triangle. We note the following.

\begin{theorem}\label{T:charactdyadtrs}
Two pointed oriented dyadic triangles are isomorphic if and only if they have the same encoding triples.
\end{theorem}

\section{Isomorphisms of dyadic triangles}\label{S:isomreprhats}

Recall that by an isomorphism between dyadic triangles we mean an isomorphism of the triangles
considered as commutative binary modes. An automorphism of the plane $\mbD^2$ is an automorphism
of the affine $\mbD$-space $\mbD^2$.
Then an isomorphism between dyadic triangles contained in $\mbD^2$ extends uniquely to an
automorphism of the affine dyadic plane $\mbD^2$. On the other hand, each automorphism $\iota$
of the plane $\mbD^2$ restricts to the isomorphism between each triangle in $\mbD^2$ and its
image under $\iota$. So, there is a one-to-one correspondence between isomorphisms of triangles
in $\mbD^2$ and automorphisms of the plane $\mbD^2$.

As each dyadic triangle is isomorphic to a representative hat, our interest focuses on
isomorphisms between pairs of representative hats.
Let $T$ be a representative hat $T_{i,j,m}$. In what follows, we use the notation $T = ABC$
precisely when   $A, B, C$ are the vertices of $T$, $A$ is the vertex located at the origin,
and the vertices $A = (0,0), B =
(i,j)$ and $C = (m,0)$ are oriented clockwise. In this situation, we say that $T_{i,j,m}$ is
\emph{presented} as $ABC$, and that $ABC$ is a \emph{presentation} of $T_{i,j,m}$.

First let us summarise necessary conditions for two representative hats to be isomorphic.

\begin{proposition}\cite[\S\S3,5]{MR25}\label{P:necconds}
If two representative hats $T = T_{i,j,m}$ and $T' = T_{i',j',m'}$ are isomorphic, then the
following two conditions hold.
\begin{enumerate}
\item[(a)] The hats $T$ and $T'$ have equal or oppositely oriented boundary types.
\item[(b)] The areas of the convex $\mbR$-hulls $\mr{conv}_{\mbR}(T)$ of $T$ and
    $\mr{conv}_{\mbR}(T')$ of $T'$ are equal.
\end{enumerate}
Moreover, if a mapping $\iota:T \rightarrow T'$ is an isomorphism, then it maps the set $V(T)$
of vertices of $T$ onto the set $V(T')$ of vertices of $T'$.
\end{proposition}

As shown in~\cite[\S\S4,5]{MR25}, none of the necessary conditions of
Proposition~\ref{P:necconds} is sufficient.
Assuming that two representative hats $T$ and $T'$ have equal or oppositely oriented  boundary
types, and the same areas of their convex $\mbR$-hulls, we
will look for some additional conditions guaranteeing that the hats are isomorphic. To avoid
misunderstanding, from now on, we denote the boundary type of a triangle $T = ABC$ by $(r,s,t)$,
where $r = \mathsf{t}(AB)$ is the type of $AB$, $s = \mathsf{t}(BC)$ of $BC$ and
$t = \mathsf{t}(CA)$ of $CA$. Recall that the type of a side does not depend on the orientation of the side, and is always a positive odd integer. If $T$ is a representative triangle, then by Pythagoras' Theorem~\cite[Thm.~4.3]{MRS11}, it follows that $r = \rm{gcd}\{i,j\}$. In particular, $r$ divides both $i$ and $j$. By~\cite[Prop.~1.8]{MMR19a}, it follows that a triple $(r,s,t)$ of positive odd integers forms a boundary type of a dyadic triangle if and only if it satisfies the condition
\begin{equation}\label{E:rst}
\rm{gcd}\{r,s\} = \rm{gcd}\{s,t\} = \rm{gcd}\{r,t\}.
\end{equation}

If $T$ is a representative hat $T_{i,j,m}$, where $i, j, m$ are odd integers, $j$ and $m$ are
positive, and $(r,s,m)$ is the boundary type of $T$, then $r$ divides both $i$ and $j$, and $s$
divides both $j$ and $m-i$. Hence $i = ar$ and $j = br = cs$ for some odd integers $a, b,
c$, where $a$ and $b$ are relatively prime.

\section{Automorphisms of representative hats}\label{S:autoreprhats}

 Consider again a representative hat $T = T_{i,j,m}$. The results of \cite[\S6]{MMR23} imply
that $T$ is generated by the union of its sides. Hence each automorphism of $T$ is determined by its restriction to the vertex set $V(T)$.
Since each automorphism of $T$ maps the vertex set to itself, and the
permutations of the vertices of a hat form the symmetric group $S_3$, it follows that $T$ may
have one, two, three or six automorphisms. Two of them correspond to  cycles of length $3$,
while three correspond to cycles of length $2$.

We start our investigation of automorphisms of representative hats with investigation of
automorphisms of certain special hats.

\subsection{Automorphisms of hats $T_{m,km,2m,0}$}\label{Sss: autmkm2mhats}

The only representative hat which is a simplex is the hat $T_{1,1,1}$, isomorphic to the hat
$T_{1,1,2,0}$ and to the right triangle $T_{1,1}$. Its boundary type is $(1,1,1)$.

\begin{lemma}\label{L:autssimplex}
The representative simplex $T_{1,1,1}$ has precisely six automorphisms, corresponding to the six permutations of its vertices.
\end{lemma}

\begin{proof}
Since the simplex $T_{1,1}$ is a free commutative binary mode, the vertices are its free
generators, and each of its automorphisms permutes the three vertices. It
follows that $T_{1,1}$, and hence also $T_{1,1,1}$, have precisely six automorphisms determined
by the permutations of the vertices.
\end{proof}

Lemma~\ref{L:autssimplex} generalizes easily to the representative hats $T_{m,m,m}$ for any
positive odd integer $m$. First note that $T_{m,m,m}$ is isomorphic
to the right hat $T_{m,m}$, which is isomorphic to $T_{0,m,2m,0}$ and to $T_{m,m,2m,0}$.

\begin{lemma}\label{L:mmm}
The representive hat $T_{m,m,m}$ has six automorphisms.
\end{lemma}

\begin{proof}
 Let the right hat $T_{m,m}$ be presented as $ABC$.
By~\cite[Lemma~3.2]{MR25}, $T_{m,m}$ has an automorphism fixing the vertex $A$ located at the
origin and exchanging the remaining two vertices. On the other hand, $T_{m,m}$ is isomorphic to the hat $T_{0,m,2m,0}$, which is isomorphic to the hat $T_{m,m,2m,0}$ presented as $T =
AB'C'$, with $B'$ the image of $B$ and $C'$ the image of $C$. By~\cite[Lemma~3.2]{MR25} again, $T$ has an automorphism fixing the vertex $B'$ and exchanging the vertices $A$ and $C'$.
It follows that the representative hat $T_{m,m,m}$ has two non-trivial automorphisms, each
corresponding to a cycle of length $2$. Since any two distinct transpositions generate the group $S_3$,
 it follows that $T_{m,m,m}$ has six automorphisms.
\end{proof}

\begin{lemma}\label{L:sidetypes}
 If a representative hat $T$ has a non-trivial automorphism,
then at least two side types of $T$ are equal.
\end{lemma}

\begin{proof}
Let $\iota$ be a non-trivial automorphism of $T$ whose restriction to the vertex set $V(T)$ forms  a cycle of length $3$. Since each automorphism of $T$ preserves the side types,
it follows that the three side types are equal. A similar argument shows that if $\iota$ is an
automorphism whose restriction to $V(T)$ is a  cycle of length $2$,  then two of the side types are equal.
\end{proof}

\begin{corollary}\label{C:difftypes}
Let $T$ be a representative  hat. If all three side types of $T$ are distinct, then
$T$ has no non-trivial automorphisms.
\end{corollary}

\begin{corollary}\label{C:twoequal}
Let $T$ be a representative hat. If two side types of $T$ are equal, and distinct from the third, then $T$ has at most one non-trivial automorphism.
\end{corollary}

\begin{proof}
By Lemma~\ref{L:sidetypes}, if $T$ has two side types equal and distinct from the third, then the restriction of a non-trivial automorphism of $T$ to the vertex set $V(T)$ cannot form a  cycle of length $3$. Thus, a non-trivial automorphism of $T$ should preserve one vertex and exchange the remaining two. Since any two  transpositions of the symmetric group $S_3$ generate a cycle of length $3$, and not all side types of $T$ are equal, it follows that $T$ may have at most one non-trivial automorphism.
\end{proof}

\begin{example}\label{Ex:rrm}
Let $T_k = AB_kC$ be a hat with vertices $A = (0,0), B_k = (m,km)$, and $C = (2m,0)$, where $k$
and $m$ are odd positive integers, with boundary type $(m,m,m)$, as considered
in~\cite[Ex.~4.3]{MR25}.
It is easy to check that the mapping $\p:(x,y) \mapsto (2m-x,y)$ is an automorphism of each such hat $T_k$. More generally, the mapping $\p$ defined in the same way
is an automorphism of each hat $T_{m,j,2m,0}$ of boundary type $(r,r,m)$, where $r =
\gcd\{m,j\}$. If $r \ne m$, then by Corollary \ref{C:twoequal}, $\p$ is the unique non-trivial
automorphism of $T_{m,j,2m,0}$. If $r = m$, then $m = \gcd\{m,j\}$. Hence $j = km$, and one
obtains the hat $T_k = AB_{k}C$.
\end{example}

\begin{proposition}\label{P:sixauts}
For any positive odd integers $m$ and $k \neq 1, 3$, the hat $T_{m,km,2m,0}$ has only one
non-trivial automorphism, while for $k = 1$ or $k = 3$, it has six automorphisms.
\end{proposition}

\begin{proof}
From Example~\ref{Ex:rrm}, we already know two automorphisms of $T_k = AB_kC$, namely the
trivial automorphism and the automorphism $\p_B = \p$, fixing $B_k$ and exchanging $A$ and $C$.

Now suppose that $T_k$ also has an automorphism $\p_A$ fixing the vertex $A$ and exchanging the vertices $B_k$ and $C$. The automorphism is given by the matrix

\begin{align*}
M_A &=
	\begin{bmatrix}
		m&km\\
		2m&0
	\end{bmatrix}^{-1}
	\begin{bmatrix}
		2m&0\\
		m&km
	\end{bmatrix}
	= \frac{-1}{2km^2}	\begin{bmatrix}
		0&-km\\
		-2m&m
	\end{bmatrix}
	\begin{bmatrix}
		2m&0\\
		m&km
	\end{bmatrix}\\ &= \frac{-1}{2km^2}	\begin{bmatrix}
	-km^2&-k^2m^2\\
	-3m^2&km^2
\end{bmatrix}= 	\begin{bmatrix}
 \frac{1}{2}&\frac{k}{2}\\
\frac{3}{2k}& -\frac{1}{2}
\end{bmatrix}
\end{align*}
with $\det M_A = -1$. However, $M_A$ is a dyadic matrix precisely when $k = 1$ or $k = 3$.
Similarly as in Lemma~\ref{L:mmm}, also in the case when $k = 3$, both $\p_B$ and $\p_A$
correspond to transpositions of vertices of $T_k$, and hence generate six different
automorphisms.

If $k \neq 1, 3$, then $\p$ is an automorphism of $T_{m,km,2m,0}$, but $\p_A$ is not.
If $T_k$ had an automorphism $\psi$ fixing the vertex $C$ and exchanging the vertices $A$ and $B_k$, or corresponding to a  cycle of length $3$, then $\p$ and $\psi$ would generate all six automorphisms of $T_k$ including $\p_A$. But this is not possible.
\end{proof}

\begin{remark}
Note that the representative form of the hat $T_{m,km,2m,0}$ may be obtained using a similar
method as in the proof of~\cite[Prop.~4.11]{MR25}.
If $k = 4l+1$, then $T_{m,km,2m,0}
\cong T_{(1+2l)m,km,m}$, while if $k = 4l+3$, then $T_{m,km,2m,0} \cong T_{(5+6l)m,km,m}$.
\end{remark}

\subsection{Automorphisms of right hats}\label{Sss: autrighthats}

The class of right hats $T_{j,m}$ is the next class of hats with easily describable automorphisms.

If the three side types of $T_{j,m}$ are distinct, then $T_{j,m}$ only has the trivial
automorphism. If $j = m$, then one obtains the hat $T_{m,m}$ isomorphic to $T_{m,m,m}$. By
Lemma~\ref{L:mmm}, it has six automorphisms. The final case to consider is the one where two side
types are equal, and distinct from the third. Assume that $T_{j,m}$ is presented as $ABC$.
Note then that $\mf{t}(BC) = \gcd\{j,m\}$ is either $j$ or $m$.

First consider the case where $j = \gcd\{j,m\}$. Then $j$ divides $m$, whence $m = kj$ for some
odd integer $k > 1$. Now $T_{j,m} = T_{j,kj}$ is isomorphic to $T_{j,j,kj}$  having the boundary type $(j,j,kj)$. By Corollary~\ref{C:twoequal}, it follows that $T_{j,j,kj}$ has at most one non-trivial automorphism. Let us call it $\p$.
It should fix the vertex $B$ and exchange the vertices $A$ and $C$. To show that such an
automorphism $\p$ exists, first translate the hat $ABC$ to the isomorphic hat $A'B'C'$ with $A' = (0,-j)$, $B' = (0,0)$ and $C' =(kj,-j)$, and then use the linear automorphism exchanging the vertices $A'$ and $C'$, given by the matrix

\begin{align*}
M_{B'} &=
	\begin{bmatrix}
		0&-j\\
		kj&-j
	\end{bmatrix}^{-1}
	\begin{bmatrix}
		kj&-j\\
		0&-j
	\end{bmatrix}
	= \frac{1}{kj^2}	\begin{bmatrix}
		-j&j\\
		-kj&0
	\end{bmatrix}
	\begin{bmatrix}
		kj&-j\\
		0&-j
	\end{bmatrix}\\ &= \frac{1}{kj^2}	\begin{bmatrix}
	-kj^2&0\\
	-k^2 j^2&kj^2
\end{bmatrix}= 	\begin{bmatrix}
 -1&0\\
-k& 1
\end{bmatrix}
.
\end{align*}
Since $\det M_{B'} = -1$, the matrix is invertible. Hence $\p$ is a non-trivial automorphism fixing $B'$ and exchanging $A'$ and $C'$.

Now assume that $\gcd\{j,m\} = m$. Then $T_{j,m} = T_{km,m}$ for some odd integer $k > 1$, and
$T_{km,m}$ is isomorphic to $T_{m,km}$ which has one non-trivial automorphism.

Summarising, we have the following proposition.

\begin{proposition}
Consider a right hat $T_{j,m}$. If the three side types of $T_{j,m}$ are distinct, then
$T_{j,m}$ has no non-trivial automorphisms. If $j = m$, then $T_{j,m}$ has six distinct
automorphisms. Otherwice, $T_{j,m}$ has one non-trivial automorphism.
\end{proposition}

\subsection{Automorphisms of hats with two side types equal}

If a representative hat $T_{i,j,m}$, with two side types equal and distinct from the third, has a
non-trivial automorphism, then it fixes the vertex common to the sides of the same type, and
exchanges the remaining vertices.

  First note an obvious but frequently used property of side types of right hats. If $T =
ABC$ is a right hat with $A =
(0,0)$, $B = (0,j)$ for some odd $j$, and $C = (n2^p,0)$ for some odd $n$ and a positive integer $p$,
then the type of $BC$ equals $\gcd\{j,n\} = \gcd\{j,n2^p\} = \gcd\{j,\pm n2^p\}$. Then note that
\eqref{E:rst} implies the following lemma.

\begin{lemma}\label{C:double}
Let $T$ be a representative hat $T_{i,j,m}$ with boundary type $(r,s,m)$.
If two side types of $T$ are equal, say to $u$, then $u$ divides the third side type.
In particular, the following hold.
\begin{enumerate}
\item[(a)] If $s = r$, then $r|m$, whence $m = ar$ for some odd integer $a$.
\item[(b)] If $r = m$, then $m|s$, whence $s = bm$ for some odd integer $b$.
\item[(c)] If $s = m$, then $m|r$, whence $r = cm$ for some odd integer $c$.
\end{enumerate}
\end{lemma}

If $ABC$ is a presentation of a representative hat $T_{i,j,m}$, then we will say
briefly that \emph{$ABC$ is the hat $T_{i,j,m}$}.

\begin{lemma}\label{L:autfixB}
Let $T = ABC$ be a representative hat $T_{i,j,m}$.
Then $T$ has an automorphism fixing the
vertex $B$ and exchanging the vertices $A$ and $C$ precisely when $j$ divides $2i - m$.
\end{lemma}

\begin{proof}
First we translate the hat $T$ to the isomorphic hat $T' = A'B'C'$ with $A' = (-i,-j)$, $B' =
(0,0)$ and $C' = (m-i, -j)$.
The hat $T'$ has an automorphism $\p_B'$ fixing the vertex $B'$ and exchanging the vertices $A'$
and $C'$ precisely when it is given by the dyadic matrix
\begin{align*}
&M_{B'} =
	\begin{bmatrix}
		-i&-j\\
		m-i&-j
	\end{bmatrix}^{-1}
	\begin{bmatrix}
		m-i&-j\\
		-i&-j
	\end{bmatrix}\\
	&= \frac{1}{jm}	\begin{bmatrix}
		-jm&0\\
		i^2 - (m-i)^2&jm
	\end{bmatrix}
	 = 	
    \begin{bmatrix}
 -1&0\\
\frac{2i - m}{j}& 1
\end{bmatrix}
.
\end{align*}
Note that $\det M_{B'} = -1$, and that $M_{B'}$ is a dyadic matrix precisely when $j$ divides
$2i-m$. Finally, we translate the hat $T'$ back to the hat $ABC$ to  obtain the required
automorphism of $T$.
\end{proof}

\begin{proposition}\label{C:rrm}
Let $T = ABC$ be a representative hat $T_{i,j,m}$ with the side types of  $AB$ and $BC$ equal to $r$. Then:
\begin{enumerate}
\item[(a)] $i = ar$ and  $j = br$ for some relatively prime odd integers $a, b$, and $m-i = cr$ for
    some even integer $c$ co-prime to $b$;
\item[(b)] $T$ has an automorphism $\p_B$ fixing the vertex $B$ and exchanging the vertices
    $A$ and $C$ precisely when $j$ divides $2i - m$, or equivalently when $b$ divides $a-c$.
\end{enumerate}
\end{proposition}

\begin{proof}
The first part follows by the fact that $r = \gcd\{i,j\} = \gcd\{m-i,j\}$, and the second by
Lemma~\ref{L:autfixB}.
\end{proof}

\begin{example}
The hat $T_{15,9,21}$ presented as $ABC$ has boundary type $(3,3,21)$. By
Proposition~\ref{C:rrm}, it has a
non-trivial automorphism fixing the vertex $B$ and exchanging the vertices $A$ and $C$.
By Corollary~\ref{C:twoequal}, this is the unique non-trivial automorphism of $T_{15,9,21}$.
\end{example}

\begin{lemma}\label{L:autfixA}
Let $T = ABC$ be a representative hat $T_{i,j,m}$. Then $T$ has an automorphism fixing the
vertex $A$ and exchanging the vertices $B$ and $C$ precisely when $m$ divides both $i$ and $j$, and $mj$ divides $m^2 - i^2$.
\end{lemma}
\begin{proof}
The hat $T$ has an automorphism $\p_A$ fixing the vertex $A$ and exchanging the vertices $B$ and $C$ precisely when it is given by the dyadic matrix
\begin{align*}
M_A &=
	\begin{bmatrix}
		i&j\\
		m&0
	\end{bmatrix}^{-1}
	\begin{bmatrix}
		m&0\\
		i&j
	\end{bmatrix}
	= \frac{-1}{jm}	\begin{bmatrix}
		0&-j\\
		-m&i
	\end{bmatrix}
	\begin{bmatrix}
		m&0\\
		i&j
	\end{bmatrix} = 	
    \begin{bmatrix}
 \frac{i}{m}&\frac{j}{m}\\
\frac{m^2 - i^2}{mj}& \frac{i}{m}
\end{bmatrix}
\end{align*}
with $\det M_A = -1$. Then $M_A$ is a dyadic matrix precisely when $m$ divides both $i$ and $j$, and $mj$ divides $m^2 - i^2$.
\end{proof}

\begin{proposition}\label{C:msm}
Let $T = ABC$ be a representative hat $T_{i,j,m}$ with the side types of $AB$ and $AC$ equal to $m$. Then
\begin{enumerate}
\item[(a)] $i = km$ and $j = lm$, where $k$ and $l$ are relatively prime odd integers;
\item[(b)] $T$ has an automorphism $\p_A$ fixing the vertex $A$ and exchanging the vertices
    $B$ and $C$ precisely when $k^2 \equiv 1 \pmod{l}$.
\end{enumerate}
\end{proposition}
\begin{proof}
Since the type $r = \gcd\{i,j\}$ of $AB$ equals $m$, it follows that $i = km$ and $j = lm$,
where $k$ and $l$ are relatively prime odd integers. By Lemma~\ref{L:autfixA}, $T$ has an
automorphism fixing $A$ and exchanging $B$ and $C$ precisely when $(m^2 - i^2)/mj = (1 - k^2)/l$ is a dyadic number. So let $(1 - k^2)/l = z 2^a$, where $a, z \in \mbZ$, $z$ is odd, and (since both $k$ and $l$ are odd) $a$ is non-negative. Set $n = - z2^a$. Then $k^2 = ln +1$ for some integer $n$, or equivalently $k^2 \equiv 1 \pmod{l}$.
\end{proof}

If the types of $AB$ and $AC$ are equal, the type of $BC$ is different, and $k^2 \equiv 1 \pmod{l}$. Then by Corollary~\ref{C:twoequal}, $\p_A$ is the unique non-trivial automorphism of $T$. Note that $1-k^2$ has no odd divisors precisely when $1-k^2 = -2^a$ for some positive integer $a$, or equivalently when $k^2 = 2^a + 1$. Since $k$ is odd, say $k = 2n + 1$, it follows that $4n^2 + 4n + 1 = 2^a + 1$, whence $n(n+1) = 2^{a-2}$.
Consequently, $a = 3$, and $k = \pm3$. For all $k$ different from $\pm3$, there are finitely
many odd positive integers $l$ such that $(1 - k^2)/l$ is a dyadic number, and infinitely many
$l$ such that $(1 - k^2)/l$ is not.

\begin{example}
The hat $T_{21,9,3}$ presented as $ABC$ has boundary type $(3,9,3)$. Since $l = 3$ divides
$1-k^2 = -48$, it follows by Proposition~\ref{C:msm} that it has a non-trivial automorphism $\p_A$ fixing the vertex $A$ and exchanging the vertices $B$ and $C$. By Corollary~\ref{C:twoequal}, this is the
unique non-trivial automorphism of $T_{21,9,3}$.
\end{example}

Remarks~$4.9$ and ~$5.2$ in~\cite{MR25} imply the following lemma.

\begin{lemma}\label{L:itom-i}
Let $T = ABC$ be a representative hat $T_{i,j,m}$. Then the hat $T$ is isomorphic to the hat $T' = C'B'A'$ with $C' = (0,0)$, $B' = (m-i,j)$  and $A' = (m,0)$, where $T'$ is the hat
$T_{m-i,j,m,0}$, and the hat $T'$ is isomorphic to the representative hat $T'' = C''B''A''$, where $T''$ is the hat
$T_{m-i+jk,j,m}$ for some odd integer $k$.
\end{lemma}

\begin{proposition}\label{C:rmm}
Let $T = ABC$ be a representative hat $T_{i,j,m}$ with the side types of $BC$ and $AC$ equal to $m$. Then
\begin{enumerate}
\item[(a)] $m-i+jk = k'm$ and $j = l'm$, where $k'$ and $l'$ are relatively prime odd
    integers;
\item[(b)] the hat $T''$ has an automorphism $\p_{C''}$ fixing the vertex $C''$ and exchanging
    the vertices $A''$ and $B''$ precisely when $(k')^2 \equiv 1 \pmod{l'}$;
\item[(c)] the hat $T$ has an automorphism $\p_{C}$ fixing the vertex $C$ and exchanging the
    vertices $A$ and $B$ precisely when $(k')^2 \equiv 1 \pmod{l'}$.
\end{enumerate}
\end{proposition}
\begin{proof}
By Lemma~\ref{L:itom-i}, $T$ is isomorphic to the representative hat $T''= C''B''A''$ with $T'' = T_{m-i+jk,j,m}$.
By Lemma~\ref{L:autfixA}, $T''$ has an automorphism $\p_{C''}$ fixing the vertex $C''$ and
exchanging the vertices $A''$ and $B''$ precisely when $m$ divides both $m-i+jk$ and $j$, and
$mj$ divides $m^2 - (m-i+jk)^2$. By Proposition~\ref{C:msm}, if $s = m$, these conditions reduce to $(k')^2 \equiv 1 \pmod{l'}$. Finally, (a) and (b) imply (c).
\end{proof}

\begin{example}\label{Ex:15,9,3}
The hat $T = T_{15,9,3}$ has the boundary type $(3,3,3)$. Since $j = 3$ divides $2i-m = 3$,
Lemma~\ref{L:autfixB} implies that $T$ has a non-trivial automorphism fixing the vertex $B$ and exchanging the vertices $A$ and $C$. Since $l = 3$ divides $1-k^2 = -24$,
Proposition~\ref{C:msm} implies that $T$ has an automorphism fixing the vertex $A$ and
exchanging
the vertices $B$ and $C$. The two automorphisms generate the remaining four automorphisms of $T$ corresponding to the appropriate permutations of the vertices of $T$. Consequently $T$ has six different automorphisms.
\end{example}

\subsection{Automorphisms of hats with three side types equal}

The last case to consider is the case of representative hats $T_{i,j,m}$ of boundary type
$(m,m,m)$ not isomorphic to $T_{m,lm,m}$ with odd $l$ (which are isomorphic to right hats) or to $T_{m,lm,2m,0}$ with odd $l > 1$. By Proposition~\ref{P:sixauts}, each hat $T_{m,3m,2m,0}$ has six automorphisms.
One can easily show that $T_{m,3m,2m,0}$ is isomorphic to the representative hat $T_{5m,3m,m}$.
Hence each hat $T_{5m,3m,m}$ has six automorphisms. (Note that in particular, the hat
$T_{15,9,3}$ of Example~\ref{Ex:15,9,3} provides an example of such a hat.)

\begin{corollary}\label{C:mmm}
Let $T = ABC$ be a representative hat $T_{i,j,m}$ with $i \neq m$, and with boundary type
$(m,m,m)$. Then
\begin{enumerate}
\item[(a)] $i = km$ and $j = lm$ for some relatively prime odd integers $k$ and $l$;
\item[(b)] $m - i = nm$ for some even integer $n$ co-prime to $l$.
\end{enumerate}
\end{corollary}

\begin{remark}\label{R:mmm}
In particular Corollary~\ref{C:mmm} implies that the hat $T$ of Corollary~\ref{C:mmm} is always crooked. If $k > m$, then $i = (n+1)m = km$, whence $n = k-1$. If $k < 0$, then $m-i = (1-k)m = nm$, whence $n = 1-k$.
\end{remark}

It is clear that if a representative hat $T = T_{i,j,m}$ of boundary type $(m,m,m)$ does not
satisfy any of the conditions (b) of Corollaries~\ref{C:rrm} and \ref{C:msm}, or (c) of
Corollary~\ref{C:rmm}, then it has no automorphism fixing (precisely) one vertex. If it
satisfies one of the conditions, then it has at least two automorphisms. If it satisfies two of the conditions, then it has six automorphisms. If $T$ has a non-trivial automorphism not fixing any of the vertices of $T$, then this   isomorphism must correspond to a cycling of the three vertices, and generates the cyclic group $C_3$. So we are left with the following question.

\begin{question}
Is there a (representative) hat with the cyclic group $C_3$ as its group of automorphisms?
\end{question}

The following proposition provides a positive answer to this question.

\begin{proposition}\label{P:mmm}
Let $T = ABC$ be a representative hat $T_{i,j,m}$ of boundary type $(m,m,m)$. Let $\p: T
\rightarrow T$ be a mapping. Then the following statements hold.
\begin{enumerate}
\item[(a)] The mapping $\p$ is an automorphism of the hat $T$ taking either $A$ to $B$, $B$ to $C$ and $C$ to $A$ or $A$ to $C$, $B$ to $A$ and $C$ to $B$ precisely when $l$ divides $k^2 - k + 1$, with $k$ and $l$ defined as in Corollary~\ref{C:mmm}.
\item[(b)] The automorphism group of $T$ contains a subgroup isomorphic to the cyclic group
    $C_3$ if and only if $l$ divides $k^2 - k + 1$.
\item[(c)] If $T$ has no automorphisms fixing precisely one vertex, then the automorphism
    group of $T$ is isomorphic to the cyclic group $C_3$ if and only if $l$ divides $k^2 - k + 1$.
\end{enumerate}
\end{proposition}
\begin{proof}
(a) First note that by Corollary~\ref{C:mmm}, $i = km$ and $j = lm$ for some relatively prime
odd integers $k$ and $l$. Then translate the hat $T$ to the isomorphic hat $T' = A'B'C'$ with
$A' = (-m,0)$, $B' = (i-m,j)$ and $C' = (0,0)$. The following matrix $M$ provides an
isomorphism from the hat $T'$ onto the hat $T$, taking $A'$ to $B$, $B'$ to $C$, and $C'$ to
$A$.
\begin{align*}
M &=
	\begin{bmatrix}
		i-m&j\\
		-m&0
	\end{bmatrix}^{-1}
	\begin{bmatrix}
		m&0\\
		i&j
	\end{bmatrix}
    =\begin{bmatrix}
	-i/m&-j/m\\
	(m^2-im+i^2)/mj&(-m+i)/m
    \end{bmatrix}\\
    &= \begin{bmatrix}
     -k&-l\\
	(1-k+k^2)/l&k-1
    \end{bmatrix}.
\end{align*}
Note that $\det M = 1$, and that the matrix $M$ is dyadic precisely when $l$ is an integer dividing
$k^2 - k + 1$.

Similarly, the matrix $M'$ below provides an isomorphism from the hat $T'$ onto the hat $T$,
taking $A'$ to $C$, $B'$ to $A$, and $C'$ to $B$.
\begin{align*}
M' &=
	\begin{bmatrix}
		-i&-j\\
		m-i&-j
	\end{bmatrix}^{-1}
	\begin{bmatrix}
		m&0\\
		i&j
	\end{bmatrix}
    =\begin{bmatrix}
	i-m/m&j/m\\
	(mi-m^2 -i^2)/mj&-i/m
    \end{bmatrix}\\
    &= \begin{bmatrix}
     k-1&l\\
	(-1+k-k^2)/l&-k
    \end{bmatrix}.
\end{align*}
The matrix $M'$ is dyadic precisely when $l$ is an integer dividing $k^2 - k + 1$.

Let $\p$ be the composition of the first translation and the isomorphism given by the matrix $M$ or $M'$. It is clear that $\p$ is an automorphism satisfying the required conditions precisely when $l$ divides $k^2 - k + 1$.

Now (b) follows directly by (a), since each version of the automorphism $\p$ generate the other and the identity map. Then (c) is a consequence of (b).
\end{proof}

\begin{example}\label{Ex:mmm1}
Consider the representative hat $T_{3,7,1}$. The boundary type of it is $(1,1,1)$. The hat has
no automorphisms fixing one vertex and exchanging the other two, since the conditions (b) of
Corollaries~\ref{C:rrm} and~\ref{C:msm}, and (c) of~\ref{C:rmm} are not satisfied. (Just note
that  $k = 3$, $l = l' = 7$ and $k' = 5$.) On the other hand, $l = 7$ divides  $k^2 -k + 1 = 7$. By Proposition~\ref{P:mmm}, the automorphism group of $T_{3,7,1}$ is isomorphic to $C_3$.
\end{example}

\begin{example}\label{Ex:mmm2}
Consider the representative hat $T = T_{21,15,3}$, with boundary type $(3,3,3)$. Using
Corollaries~\ref{C:rrm} and \ref{C:msm}, one can show that $T$ has no non-trivial automorphism
fixing the vertex $A$ or the vertex $B$.  By Corollary~\ref{C:rmm}, the hat $T$ has a
non-trivial automorphism fixing the vertex $C$. Indeed, $m-i+j = (-1)3 = k'm$, whence $k'= -1$
and $j = 15 = 5\cdot3 = l'm = lm$, with $l = l' = 5$. It follows that $(k')^2 \equiv 1
\pmod{l'}$. On the other hand, by Proposition~\ref{P:mmm}, the hat $T$ does not have a
non-trivial automorphism corresponding to a $3$-element cycle of the vertices, since $l = 5$
does not divide $k^2-k+1 = 43$. It follows that the automorphism group of $T_{21,15,3}$ is
isomorphic to the symmetric group $C_2$.
\end{example}

\section{Isomorphisms of representative hats}\label{S:isomorphisms}

First recall that, by Proposition~\ref{P:necconds}, an isomorphism
between two hats maps the set of vertices of one hat onto the set of vertices of the other, and that isomorphic hats have equal or oppositely
oriented boundary types and equal areas of their convex $\mbR$-hulls. Note as well that an
isomorphism between two hats preserves the types of the sides.
As the following example shows, all these conditions together do not guarantee that two hats are isomorphic.

\begin{example}
The representative hats $T_{3,27,21}$ and $T_{39,27,21}$ both have the boundary type $(3,9,21)$, and the area of the convex $\mbR$-hull of each equals $21\cdot27/2$.
However, by~\cite[Prop.~5.1]{MR25},
they are not isomorphic.
\end{example}

We will use the following consequence of~\cite[Rem.~4.10, Prop.~4.11, Rem.~4.12]{MR25}.

\begin{proposition}\label{P:equivrepr}
Let $T = ABC$ be the hat with the vertices $A = (0,0)$, $B = (i,j)$ and $C = (m,0)$, where $i$ is any integer and $j, m$ are positive integers.
 Then $T$ may be specified, in terms of the parameters $i,j,m$, in each of the following three
ways:
\begin{enumerate}
\item[(a)] as $T_{i,j,m,0}$ with odd $j$ and odd $\gcd\{i,m\}$,
\item[(b)] as $T_{i,j,m} = T_{i,j,m,0}$, where all $i, j, m$ are odd,
\item[(c)] as $T_{i,j,m,0}$ with even $i$, odd $j$ and odd $m$.
\end{enumerate}

\end{proposition}

A hat $T_{i,j,m,0}$, where $i$ is any integer, and both $j$ and $m$ are positive odd integers,
will be called \emph{almost representative} and denoted $\ovT_{i,j,m}$. If $i$ is odd, then
$T_{i,j,m,0} = T_{i,j,m}$ is representative (or \emph{odd representative}), and if $i$ is even, then $T_{i,j,m,0}$ is \emph{even representative}. By Remark~$4.9$ of~\cite{MR25}
and Proposition~\ref{P:equivrepr}, the hats $\ovT_{i,j,m}$ and $\ovT_{m-i,j,m}$ are isomorphic, and if one of them is representative, then the other is even representative. This isomorphism of $\ovT_{i,j,m}$ and $\ovT_{m-i,j,m}$ will be denoted by $\kappa$.

To make further calculations easier we will use the following notation. Let $T$ and $T'$ be two almost representative hats. If $\iota: T \rightarrow T'$ is a mapping taking the set $V(T)$ of vertices of $T$ onto the set $V(T')$ of vertices of $T'$, then the images of the vertices $A, B, C$ of $T$ will be denoted by the same primed letters $A', B', C'$, respectively. If $T = ABC$,
then $T' = XYZ$ where $XYZ$ is a permutation of the images $A', B', C'$, $X = (0,0)$ and $X, Y, Z$ are oriented clockwise. This fact will be denoted briefly as $\iota: ABC \rightarrow XYZ$. Note also that there are six types of isomorphisms between $T$ and $T'$ corresponding to the six permutations of the vertices $A', B'$ and $C'$ of $T'$.

\begin{theorem}\label{T:arhathat}
Let $T$ and $T'$ be two almost representative hats. Let $T$ be a hat $\ovT_{i,j,m}$ with
vertices $A, B, C$, and $T'$ be a hat $\ovT_{k,l,n}$ with vertices $A', B', C'$. Assume that
both hats have equal or oppositely oriented boundary types and equal areas of their convex
$\mbR$-hulls. Let $\iota:T \rightarrow T'$ be a mapping taking the set $V(T)$ of vertices of $T$ onto the set $V(T')$ of vertices of $T'$. Then one of the following six statements
holds.
\begin{enumerate}
\item[(a)] If $\iota: ABC \rightarrow A'B'C'$, then $\iota$ is an isomorphism  between the hats $T$ and $T'$  precisely when $l = j, n = m$ and $k = i + jp$ for some integer $p$.
\item[(b)] If $\iota: ABC \rightarrow C'B'A'$, then $\iota$ is an isomorphism  between the
    hats $T$ and $T'$  precisely when $l = j, n = m$ and $k = m-i + jp$ for some integer
    $p$.
\item[(c)] If $\iota: ABC \rightarrow A'C'B'$, then $\iota$ is an isomorphism  between the
    hats $T$ and $T'$  precisely when
    $n = \gcd\{i,j\}$, $l = mj/n$ and $k = am$, where $a$ is a solution of the equation $ai
    \equiv n \pmod{j}$.
\item[(d)] If $\iota: ABC \rightarrow C'A'B'$, then $\iota$ is an isomorphism between the
    hats $T$ and $T'$  precisely when
    $n = \gcd\{m-i,j\}$, $l = mj/n$ and $k = am$, where $a$ is a solution of the equation
    $a(m-i) \equiv n \pmod{j}$.
\item[(e)] If $\iota: ABC \rightarrow B'C'A'$, then $\iota$ is an isomorphism  between the
    hats $T$ and $T'$  precisely when
    $n = \gcd\{i,j\}$, $l = mj/n$ and $k = n-am$, where $a$ is a solution of the equation $ai \equiv n \pmod{j}$.
\item[(f)] If $\iota: ABC \rightarrow B'A'C'$, then $\iota$ is an isomorphism  between the
    hats $T$ and $T'$  precisely when
    $n = \gcd\{m-i,j\}$, $l = mj/n$ and $k = n-am$, where $a$ is a solution of the equation
    $a(m-i) \equiv n \pmod{j}$.
\end{enumerate}
\end{theorem}
\begin{proof}
First note that since the areas of the convex $\mbR$-hulls of $T$ and $T'$ are equal, it follows that
\begin{equation}\label{E:equalareas}
mj = nl.
\end{equation}
Since the boundary types of $T$ and $T'$ are equal or oppositely oriented, it follows that
the types of each side of $T$ and the side of $T'$ that is its image under $\iota$ are equal.

(a) This statement follows by~\cite[Prop.~5.1]{MR25}
and Proposition~\ref{P:equivrepr}.

(b) This statement follows by~\cite[Rem.~5.2]{MR25}
and Proposition~\ref{P:equivrepr}.
Note the role of the isomorphism $\kappa$.

(c) A matrix
\begin{equation}
M =
\begin{bmatrix}
	a&b\\
	c&d
\end{bmatrix}
\end{equation}
is the matrix of a linear isomorphism of the $\mbD$-space $\mbD^2$ when it is an
invertible dyadic matrix. Then $M$ maps the points $A, B, C$ to the points $A', C', B'$
respectively precisely when
\begin{align*}
&
\begin{bmatrix}
		m&0\\
		i&j
	\end{bmatrix}
	\begin{bmatrix}
		a&b\\
		c&d
	\end{bmatrix}
    =\begin{bmatrix}
	am&bm\\
	ai+cj&bi+dj
    \end{bmatrix}\\
    &= \begin{bmatrix}
     k&l\\
	 n&0
    \end{bmatrix}.
\end{align*}
The latter equality holds precisely when
\begin{equation}
k = am \ \mbox{ and } \ l = bm,
\end{equation}
and
\begin{equation}
ai + cj = n\ \mbox{ and } \ bi + dj = 0.
\end{equation}
However $mj = nl$ and $n = \gcd\{i,j\}$. Hence $b = l/m = j/n \in \mbZ$.
Then $ai +cj = n = \gcd\{i,j\}$ has an integer solution. Alos, $bi + dj = 0$ implies that
$d = -bi/j = -i/n \in \mbZ$. One easily sees that
\begin{equation}
M =
\begin{bmatrix}
	a&j/n\\
	(n-ai)/j&-i/n
\end{bmatrix}
\end{equation}
has integers as entries and $\det M = -1$. Hence $M$ is an invertible dyadic matrix.
Consequently, $M$ is the matrix of the isomorphism $\iota = \iota^c: T \rightarrow T'$ precisely when $a$ is a solution of the equation $ai \equiv n \pmod{j}$.

(d) In this case we first use the isomorphism $\kappa$ from $\ovT_{i,j,m}$ to $\ovT_{m-i,j,m}$, and
then apply the isomorphism $\iota^c$ of the case (c) to $\ovT_{m-i,j,m}$. The composition of
these two isomorphisms provides the isomorphism $\iota:T \rightarrow T'$, where $T'$ is
presented as $C'A'B'$. In particular, $\iota$ is an isomorphism precisely when $n =
\gcd\{m-i,j\}$, $l = mj/n$ and $k = am$, where $a$ is a solution of the equation $a(m-i) \equiv n \pmod{j}$.

(e) Now we first use the isomorphism $\iota^c$ of (c), and then apply the isomorphism $\kappa$
to the image. The composition of these two isomorphisms provides the isomorphism $\iota:T
\rightarrow T'$, where $T'$ is presented as $B'C'A'$ with $C' = (n-k,l)$. The mapping $\iota$ is an isomorphism precisely when
$n = \gcd\{i,j\}$, $l = mj/n$ and $k = n-am$, where $a$ is a solution of the equation $ai \equiv n \pmod{j}$.

(f) In this case the isomorphism $\iota:T \rightarrow T'$ is obtained as a composition of the
isomorphism of the case (d) with the isomorphism $\kappa$ applied to the image. Then $T'$ is
presented as $B'A'C'$ with $A' = (n-am,l)$. The mapping $\iota$ is an isomorphism precisely when
$n = \gcd\{m-i,j\}$, $l = mj/n$ and $k = n-am$, where $a$ is a solution of the equation $a(m-i) \equiv n \pmod{j}$.
\end{proof}

\begin{example}
Consider the representative hat $T_{1,3,5}$ presented as $T = ABC$. Then Theorem~\ref{T:arhathat} provides the
following examples of (representative) hats $T'$ isomorphic to $T$ in each of the six cases considered. In case (a), this is e.g. the hat $T_{4,3,5}$ presented as $A'B'C'$. In the case (b), this may be $T_{7,3,5}$ presented as $C'B'A'$. In both the cases (c) and (d), one can find the hat $T_{5,15,1}$ presented as $A'C'B' = C'A'B'$, and in both the cases (e) and (f), the hat $T_{11,15,1}$ with the presentation $B'C'A' = B'A'C'$.

\begin{figure}[H]
	\begin{center}
		\begin{picture}(260,180)
			
			\put(0,20){\vector(1,0){140}}
			\put(20,0){\vector(0,1){100}}
			\put(20,20){\circle*{5}}
			
			\put(120,20){\circle*{5}}
			\put(40,80){\circle*{5}}
			\put(20,20){\line(1,3){20}}
			\put(120,20){\line(-4,3){80}}

			\put(7,7){$A$}
			\put(44,80){$B  $}
			\put(119,23){$C $}
			\put(120,9){$5$}
			\put(18,80){\line(1,0){4}}
			\put(8,78){$3$}
	
			\put(150,20){\vector(1,0){80}}
			\put(170,0){\vector(0,1){185}}
			\put(170,20){\circle*{5}}
			\put(183,20){\circle*{5}}

			\put(170,20){\line(1,3){50}}
	 \put(183,20){\line(1,4){37}}

			\put(220,170){\circle*{5}}	
			\put(137,7){$A'=C'$}
			\put(224,170){$C'=A'$}
			\put(185,24){$B'$}
		\put(185,9){$1$}
			\put( 220,18){\line(0,1){4}}
		\put(220,8){$5$}
			\put(168,170){\line(1,0){4}}
		\put(156,167){$15$}
		\end{picture}
	\end{center}
	\caption{The isomorphic hats $T_{1,3,5}$ and $T_{5,15,1}$.}
	\label{F:2}
\end{figure}
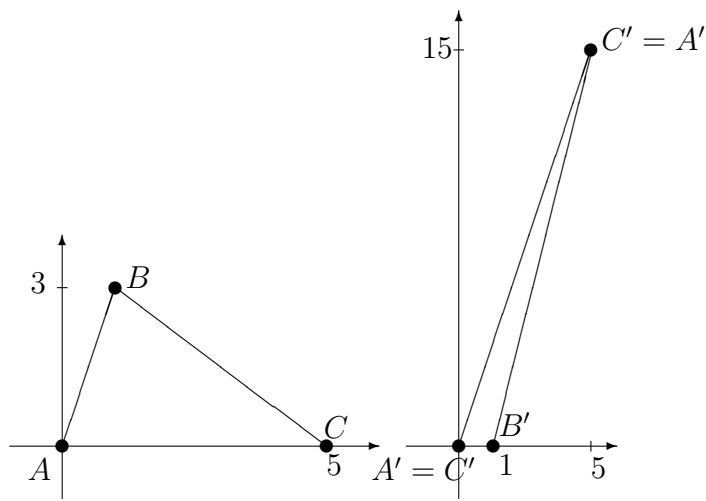
\end{example}


\begin{thebibliography}{99}

\bibitem{AB83}
A. Br\o ndsted,
    {\em An Introduction to Convex Polytopes,}
    Springer Verlag, New York, 1983.\\
\texttt{https://doi.org/10.1007/978-1-4612-1148-8}

\bibitem{C75}
B. Cs\'{a}k\'{a}ny,
Varieties of affine modules,
{\em Acta Sci. Math.} {\bf 37} (1975) 3--10.

\bibitem{CR13}
G. Cz\'{e}dli, A. Romanowska,
Generalized convexity and closure conditions,
{\em Int. J. of Algebra and Computation} {\bf 23} (2013) 1805--1835.
\\ \texttt{https://doi.org/10.1142/S0218196713500458}

\bibitem{BG03}
B. Gr\"{u}nbaum,
    {\em Convex Polytopes,}
    2nd ed., Springer Verlag, New York, 2003.
\\ \texttt{ https://doi.org/10.1007/978-1-4613-0019-9 }

\bibitem{MMR19}
K. Matczak, A. Mu\'{c}ka, A. Romanowska,
Duality for dyadic intervals,
{\em Int. J. of Algebra and Computation},  {\bf 29} (2019) 41--60.
\\ \texttt{ https://doi.org/10.1142/S0218196718500625 }

\bibitem{MMR19a}
K. Matczak, A. Mu\'{c}ka, A. Romanowska,
Duality for dyadic triangles,
{\em Int. J. of Algebra and Computation}, {\bf 29} (2019) 61--83.
\\ \texttt{ https://doi.org/10.1142/S0218196718500637 }

\bibitem{MMR23}
K. Matczak, A. Mu\'{c}ka, A. Romanowska,
Finitely generated dyadic convex sets,
{\em Int. J. of Algebra and Computation}, {\bf 33} (2023) 585--615.
\\ \texttt{ https://doi.org/10.1142/S0218196723500273 }

\bibitem{MRS11}
K. Matczak, A. Romanowska, J. D. H. Smith,
   Dyadic polygons,
   {\em Int. J. of Algebra and Computation} {\bf 21} (2011) 387--408.
\\ \texttt{ https://doi.org/10.1142/S0218196711006248 }

\bibitem{MR25}
A. Mu\'{c}ka, A. Romanowska,
Geometry of dyadic polygons I: The atructure of dyadic triangles,
preprint (2025).


\bibitem{R18}
A. Romanowska,
Convex sets and barycentric algebras,
in: Nonassociative mathematics and its applications, 243-–259, {\em Contemp. Math.} {\bf 721},
Amer. Math. Soc., Providence, RI, 2019.
\texttt{https://doi.org/10.1090/conm/721/14509}


\bibitem{RS85}
A. B. Romanowska, J. D. H. Smith,
    {\em Modal Theory,}
    Heldermann Verlag, Berlin, 1985.


\bibitem{RS02}
A. B. Romanowska, J. D. H. Smith,
    {\em Modes,}
    World Scientific, Singapore, 2002.
\texttt{https://doi.org/10.1142/4953}

\bibitem{Z95}
G. M. Ziegler,
{\em Lectures in Polytopes},
Springer-Verlag, New York, 1995.
\\ \texttt{ https://doi.org/10.1007/978-1-4613-8431-1 }
\end{thebibliography}
\end{document}